\documentclass[reqno,intlimits,twoside]{article}

\usepackage{fancyhdr}
\usepackage{amsmath}
\usepackage{amssymb}
\usepackage{cite}
\usepackage{amsthm}
\usepackage{color}

\newtheorem{lemma}{Lemma}[section]
\newtheorem{theorem}{Theorem}[section]

\newtheorem*{proposition}{Proposition}
\theoremstyle{definition}
\newtheorem{definition}{Definition}[section]

\newtheorem{remark}{Remark}[section]
\newtheorem{question}{Question}[section]

\numberwithin{equation}{section}

\begin{document}

\thispagestyle{empty}

\begin{center}
    \mbox{\hskip -1.2cm \fontsize{12}{14pt}\selectfont
    \textbf{A new gap in the critical exponent for semi-linear structurally damped evolution equations }} \\[0mm]
\end{center}

\medskip

\ \!\!\!\hrulefill \

\begin{center}
\textbf{\fontsize{11}{14pt}\selectfont Khaldi Said, Arioui Fatima Zahra and Hakem Ali}
\end{center}
\vskip+1cm
\thispagestyle{empty}

\noindent \textbf{Abstract.}
Our aim in this paper is to discuss the critical exponent in semi-linear structurally damped wave and beam equations with additional dispersion term. The special model we have in mind is

$$
u_{tt}(t,x)+(-\Delta)^{\sigma}u(t,x)+(-\Delta)^{2\delta}u(t,x)+2(-\Delta)^{\delta}u_{t}(t,x)=\left|u(t,x)\right| ^{p}
$$ 
where the initial displacement $u(0,x)=u_{0}(x)$, the initial velocity $u_{t}(0,x)=u_{1}(x)$ and the parameters $ t\in [0,\infty)$, $x\in \mathbb{R}^{n}$, $\sigma\geq 1$, $\delta\in(0,\frac{\sigma}{2})$, $p>1$. 

The solution to the linear equation at low frequency region involves an interplay of diffusion and oscillation phenomena represented by a real-complex Fourier multiplier of the form 
 $$m(t,\xi)=\frac{e^{-|\xi|^{2\delta}t\pm i|\xi|^{\sigma}t}}{2i|\xi|^{\sigma}}, \ \ \xi\in \mathbb{R}^{n}, \ \ i=\sqrt{-1}.$$ 

The scaling argument shows that the diffusive part leads to faster decay rates compared to the oscillatory one. This interplay creates a new gap in the critical exponent between the blow up (in finite time) result when $1<p<1+\frac{4\delta}{n-2\delta}$ (sub-critical case) and the global (in time) existence result
when $p>1+\frac{\sigma+2\delta}{n-\sigma}$ (super-critical case).

We leave an open to show if this gap will be closed at least in low or high space dimensions because, to the best of authors knowledge, the necessary Fourier multiplier that leads to the sub-critical case does not explicitly appear in $m(t,\xi)$.
\vskip+0.5cm

\noindent \textbf{2010 Mathematics Subject Classification.}
    42B10, 35G10, 35B45

\medskip

\noindent \textbf{Key words and phrases.} Wave equation, beam equation, structural damping, dispersion terms, global existence, blow up, critical exponent.

\vskip+0.4cm

 \newpage
\pagestyle{fancy}
\fancyhead{}
\fancyhead[EC]{\hfill \textsf{\footnotesize Khaldi Said et al}}
\fancyhead[EL,OR]{\thepage}
\fancyhead[OC]{\textsf{\footnotesize A new gap in the critical exponent for semi-linear structurally damped evolution equations}\hfill }
\fancyfoot[C]{}
\renewcommand\headrulewidth{0.5pt}
%-------------------------------------------------------------------------------
\section{Introduction}
Many interesting phenomena in nature can be characterized as diffusion, oscillation or coupled diffusion-oscillation due to the presence of real and complex characteristic roots in formal solutions to the linear evolution equations modeling them, as is well known from a mathematical Fourier analysis and physical perspective. For instance, a first-order evolution equation of anomalous diffusion and a second-order evolution equation of waves being particularly most important. The Schro\"{o}dinger equation is an exception, since it is a first-order evolution equation that explicitly includes the complex factor $i$. More precisely, let us distinguish between the above three phenomena that appear in the formal solution to some linear evolution models listed below:
\begin{itemize}
	\item Diffusion, which means that the formal solution contains a real Fourier multiplier of the form $e^{-c|\xi|^{a}t}$, where $t>0$, $\xi\in \mathbb{R}^{n}$ and  $a, c$ are some strictly positive constants. For example, this phenomenon appears in the following models:
	\begin{itemize}
		\item Anomalous diffusion equation(heat equation when $\sigma=1$)
		\begin{equation}\label{1.1}
		v_{t}(t,x)+c(-\Delta)^{\sigma}v(t,x)=0,\  \ \  v(0,x)=v_{0}(x), \ \ \ \sigma>0,
		\end{equation}
		the formal solution in the Fourier analysis is given by
		\begin{equation*} \hat{v}(t,\xi)=e^{-c|\xi|^{2\sigma}t}\hat{v_{0}}(\xi), 
		\end{equation*}
		which is valid for all frequency regions,  for more details see  \cite{EbertReissig, vazquez, dabbicco-ebert-diffusion phenomenon} and reference therein. 
		\item Wave or beam equations with frictional damping (waves when $\sigma=1$, beam when $\sigma=2$),
		\begin{equation}\label{1.2}
		u_{tt}(t,x)+(-\Delta)^{\sigma}u(t,x)+u_{t}(t,x)=0,\  u(0,x)=0, \ u_{t}(0,x)=u_{1}(x), 
		\end{equation}
		the formal solution in the Fourier analysis is given by
		\begin{equation*}
		\hat{u}(t,\xi)=\frac{e^{-\frac{|\xi|^{2\sigma}}{1+\sqrt{1-4|\xi|^{2\sigma}}}t}-e^{-t-\sqrt{1-4|\xi|^{2\sigma}}t}}{\sqrt{1-4|\xi|^{2\sigma}}}\hat{u_{1}}(\xi)
		\end{equation*}
	which is valid at low frequency region $\{\xi \in \mathbb{R}^{n} : |\xi|<(1/4)^{1/2\sigma}\}$, for more details see  \cite{fujiwara-ikeda-wakasugi, duong reissig, matsumura 1, matsumura 2} and reference therein. 
		\item Wave or beam equations with effective structural damping
		\begin{equation}\label{1.3}
		u_{tt}(t,x)+(-\Delta)^{\sigma}u(t,x)+(-\Delta)^{\delta}u_{t}(t,x)=0,\  u(0,x)=0, \ u_{t}(0,x)=u_{1}(x), 
		\end{equation}
		the formal solution in the Fourier analysis is given by
		\begin{equation*}
		\hat{u}(t,\xi)=\frac{e^{-\frac{|\xi|^{2\sigma-2\delta}}{1+\sqrt{1-4|\xi|^{2\sigma-4\delta}}}t}-e^{-|\xi|^{2\delta}t-|\xi|^{2\delta}\sqrt{1-4|\xi|^{2\sigma-4\delta}}t}}{|\xi|^{2\delta}\sqrt{1-4|\xi|^{2\sigma-4\delta}}}\hat{u_{1}}(\xi), \ \delta \in (0,\sigma/2)
		\end{equation*}
		which is valid at low frequency $\{\xi \in \mathbb{R}^{n} :|\xi|<(1/4)^{1/(2\sigma-4\delta)}\}$, for more details see \cite{dabbicco-ebert-a classification,dabbicco-ebert-new phenomenon,pham-reissig-kainan} and reference therein. 
		\item Wave or beam equations with effective critical structural damping
		\begin{equation}\label{1.4}
		u_{tt}(t,x)+(-\Delta)^{\sigma}u(t,x)+\mu(-\Delta)^{\sigma/2}u_{t}(t,x)=0,\  u(0,x)=0, \ u_{t}(0,x)=u_{1}(x), 
		\end{equation}
		the formal solution in the Fourier analysis is given by
		\begin{equation*}
		\hat{u}(t,\xi)=\frac{e^{-\frac{|\xi|^{\sigma}}{\mu+\sqrt{\mu^{2}-4}}t}-e^{-\mu/2|\xi|^{\sigma}t-1/2|\xi|^{\sigma}\sqrt{\mu^{2}-4}t}}{|\xi|^{\sigma}\sqrt{\mu^{2}-4}}\hat{u_{1}}(\xi), \ \ \mu\in (2,\infty)
		\end{equation*}
		or
		\begin{equation*}\label{1.5}
		\hat{u}(t,\xi)=te^{-|\xi|^{\sigma}t}\hat{u_{1}}(\xi), \ \ \mu=2,
		\end{equation*}
		for more details see \cite{dabbicco-ebert-a classification,dabbicco-ebert-new phenomenon,pham-reissig-kainan} and reference therein.
	\end{itemize}
	\item Oscillations, which means that the formal solution contains a complex Fourier multiplier of the form $e^{-ic|\xi|^{a}t}$. For example, this phenomenon appears in the following models:
	\begin{itemize}
		\item Schr\"{o}dinger equation, see \cite{EbertReissig} and reference therein,
		\begin{equation}\label{1.5}
		v_{t}(t,x)+i(-\Delta)^{\sigma}v(t,x)=0,\  v(0,x)=v_{0}(x),
		\end{equation}
		the formal solution in the Fourier analysis is given by
		\begin{equation*}
		 \hat{v}(t,\xi)=e^{-i|\xi|^{2\sigma}t}\hat{v_{0}}(\xi).
		\end{equation*}
		\item Free wave and beam equations, see \cite{ebert-Lourenco} and reference therein,
		\begin{equation}\label{1.6}
		u_{tt}(t,x)+(-\Delta)^{\sigma}u(t,x)=0,\  u(0,x)=0, \ u_{t}(0,x)=u_{1}(x), 
		\end{equation}
		the formal solution in the Fourier analysis is given by
		\begin{equation*}
		 \hat{u}(t,\xi)=\frac{e^{-i|\xi|^{\sigma}t}-e^{i|\xi|^{\sigma}t}}{2i|\xi|^{\sigma}}\hat{u_{1}}(\xi).
		\end{equation*}
	\end{itemize}
	\item Coupled diffusion-oscillation, which means that the formal solution contains both real and complex Fourier multipliers of the form $e^{-c|\xi|^{a}t\pm i|\xi|^{b}t}$ with $a, b, c>0$. This case itself contains three different phenomena as follows:
	\begin{itemize}
		\item If $a<b$, diffusion is more dominant than oscillation, in the sense that the scaling argument for the diffusive part leads to a better decay rate than that for the oscillatory part. For example,
		\begin{itemize}
			\item [-]Generalized Schr\"{o}dinger equation
			\begin{equation}\label{1.7}
			v_{t}(t,x)+(-\Delta)^{2\delta}v(t,x)+i(-\Delta)^{\sigma}v(t,x)=0,\  u(0,x)=v_{0}(x), \ \ 2\delta<\sigma.
			\end{equation}
			\item [-] Wave and beam equations with structural damping and additional dispersion term (our interest)
			\begin{equation}\label{1.8}
			u_{tt}(t,x)+(-\Delta)^{2\delta}u(t,x)+(-\Delta)^{\sigma}u(t,x)+2(-\Delta)^{\delta}u_{t}(t,x)=0, \ \ 2\delta<\sigma, 
			\end{equation}
			$$u(0,x)=0, \ u_{t}(0,x)=u_{1}(x),$$
			the formal solution in the Fourier analysis is given by
			\begin{equation*}
			\hat{u}(t,\xi)=\frac{e^{-|\xi|^{2\delta}t+i|\xi|^{\sigma}t}-e^{-|\xi|^{2\delta}t-i|\xi|^{\sigma}t}}{2i|\xi|^{\sigma}}\hat{u_{1}}(\xi).
			\end{equation*}
		\end{itemize}
		\item If $a=b$ then the diffusion plays the same role as oscillation in the sense that they lead to the same decay rate.
		\item If $a>b$, oscillation is more dominant than diffusion, using the same above sense. We remark that this phenomenon can be described again by models (\ref{1.7})-(\ref{1.8}) when $\sigma<2\delta$ or by the following model:
		\begin{itemize}
			\item[-] Wave and beam equations with non-effective structural damping
			\begin{equation}\label{1.9}
			u_{tt}(t,x)+(-\Delta)^{\sigma}u(t,x)+(-\Delta)^{\delta}u_{t}(t,x)=0,\  u(0,x)=0, \ u_{t}(0,x)=u_{1}(x),
			\end{equation}
			the formal solution in the Fourier analysis is given by
			\begin{equation*}
			\hat{u}(t,\xi)=\frac{e^{-|\xi|^{2\delta}t+i|\xi|^{\sigma}\sqrt{1-4|\xi|^{4\delta-2\sigma}}t}-e^{-|\xi|^{2\delta}t-i|\xi|^{\sigma}\sqrt{1-4|\xi|^{4\delta-2\sigma}}t}}{2i|\xi|^{\sigma}\sqrt{1-4|\xi|^{4\delta-2\sigma}}}\hat{u_{1}}(\xi),
			\end{equation*}
			provided that  $\sigma<2\delta<2 \sigma$ and $ |\xi|<(1/4)^{1/(4\delta-2\sigma)}$. For more details see \cite{dabbicco-ebert-a classification,pham-reissig-kainan,dabbicco-ebert-lp-lq, dabbicco-ebert-critical exponent} and reference therein.
		\end{itemize}	
	\end{itemize}
\end{itemize} 

We have now discussed all possible scenarios of interaction between diffusion and oscillation in linear dynamical systems. To the best of authors knowledge, the dynamical systems that satisfy pure diffusion at low frequency region always achieve optimal results, not only in terms of the decay rate but also in the critical exponent to the corresponding semi-linear evolution models. On the contrary, the situation pose some challenges in those of pure oscillation or coupled diffusion-oscillation. 

After the seminal paper \cite{fujita} of H. Fujita in 1966 about the semilinear heat equation, several mathematicians focus their research works to investigate both global (in time) existence of small data solutions and blow up (in finite time) of weak solutions to some semi-linear evolution equations with different power nonlinearities in the right hand side, in particular, to find one of the most important values known as \textit{Fujita critical exponent} denoted by $p^{Fuj}$. In general, critical exponent is exactly a threshold that divides the range of power nonlinearities $p\in(1,\infty)$ into two parts such that when:
\begin{itemize}
	\item $1 < p <p^{Fuj}$(sub-critical case), there exist arbitrarily small Cauchy data, such that there exists no global (in time) weak solution, only local existence can be proved (so-called blow-up result). 
	\item $p^{Fuj}<p<\infty$(super-critical case), there exist unique global (in time) small data Sobolev or energy solutions.
\end{itemize} 
The critical case $p=p^{Fuj}$ belongs either to the set of global existence or blow-up according to the regularity of initial data. Let us now review some previous critical exponents to the above mentioned evolution equations when the usual power nonlinearity $|u(t,x)|^{p}$ is applied as a source term for all $p>1$.

 It has been shown that the critical exponent remains the same in both models (\ref{1.1}) and (\ref{1.2}) and is given by

\begin{equation}\label{1.11}
p_{1}^{Fuj}(n,\sigma,m)=1+\frac{2m\sigma}{n}, \ \ m\in [1,2], \ n\geq 1,
\end{equation}
where the parameter $m$ denotes the additional $L^{m}$ regularity of initial data $v_{0}, u_{1} \in L^{m}\cap L^{2}$. 

The critical exponent to model (\ref{1.3}) is given by 
\begin{equation}\label{1.12}
p_{2}^{Fuj}(n,\sigma, \delta)=1+\frac{2\sigma}{n-2\delta},\ \ n>2\delta,
\end{equation}
provided that $u_{1} \in L^{1}\cap L^{2}$. The critical exponent to the special model (\ref{1.4}) is given by 
\begin{equation}\label{1.13}
p_{2}^{Fuj}(n,\sigma, \sigma/2)=1+\frac{2\sigma}{n-\sigma},\ \ n\geq 1.
\end{equation}
Now, back to the models (\ref{1.6}) and (\ref{1.9}). In these two cases, the authors cannot determine the exact critical exponent for all space dimension as in the diffusive case, but there are some positive partial results at least in low space dimension.

Let us begin with (\ref{1.6}), indeed, using $L^{q}$-energy estimates $(q \neq 2)$ together with Banach fixed point theorem and test function method, the authors in \cite{ebert-Lourenco} successfully proved that the critical exponent is still given by (\ref{1.13}) only at low space dimension $1<\sigma<n\leq2\sigma$, however, in the case $n> 2\sigma$, the authors said that we may have a gap. 

The situation is more difficult in model (\ref{1.9}). In fact, on the one hand, if we use $L^{2}$-energy estimates together with Banach fixed point theorem we get that the global (in time) existence holds for any $p$ strictly larger than $p_{0}$, where
$$p_{0}:=p_{0}(n,\sigma,\delta)=1+\frac{\sigma+2\delta}{n-\sigma}, \ \ \delta\in \left(\frac{\sigma}{2},\sigma\right).$$ 

On the other hand, if we use a test function method, we prove a  blow up in finite time of weak solutions for any $p$ strictly smaller than $p_{2}^{Fuj}(n,\sigma, \sigma/2)$.
 
These techniques show a gap between non-existence and existence results in the following interval
$$\left(1+\frac{2\sigma}{n-\sigma}, 1+\frac{\sigma+2\delta}{n-\sigma}\right).$$ 

Fortunately, the authors in \cite{dabbicco-ebert-critical exponent,dabbicco-ebert-lp-lq} have succeeded recently derived sharp auxiliary $L^{q}$-estimates and found the critical exponent given by $p_{2}^{Fuj}(n,\sigma, \sigma/2)$ but again at low space dimension  $1<\sigma<n<\bar{n}(\sigma)$ as in model (\ref{1.6}), where,
$$\bar{n}(\sigma)=(3\sigma-2)\left[ 1+\frac{1}{2}\left( \sqrt{1+8\sigma(3\sigma-2)^{-2}}-1\right) \right] \in (3\sigma-2, 3\sigma-1). $$ 

Thus, the gap in the above interval was partially closed. Of course, this result is reasonable thanks to the presence of the same complex Fourier multiplier $|\xi|^{-\sigma}e^{\pm i|\xi|^{\sigma}t}$ in all models (\ref{1.4}), (\ref{1.6}) and (\ref{1.9}) which really leads to the unified sub-critical case $p_{2}^{Fuj}(n,\sigma, \sigma/2)$.

Now, inspired by those interesting considerations \cite{dabbicco-ebert-lp-lq,dabbicco-ebert-critical exponent} we are motivated to ask the following question:

\begin{question}
	What happens if the scaling argument of the diffusive part is better than the oscillatory one?
\end{question}  

To answer this question, we will focus our attention to derive some $L^{2}$-energy estimates and investigate both global (in time) existence of small data solution as well as a blow up (in finite time) of weak solution to the special model (\ref{1.8}). We found a new gap, similar to the one in previous model (\ref{1.9}) and we have a second question:

\begin{question}
	Can this gap be closed, at least in low or high dimensions, as in (\ref{1.9})?
\end{question}

Now, we are in position to state our principal results in the following section and prove them in the next Sections \ref{section 3}, \ref{section 4}. We will choose $u(0,x)=0$ for brevity.
\section{Principal Results}\label{section 2}

Our first principal result concerning the global (in time) existence of small data solutions reads as follows.
\begin{theorem}\label{GlobalExistence1}
	Let $t\in \mathbb{R}_{+}$, $x\in \mathbb{R}^{n}$, $\sigma\geq 1$, $\delta\in \left(0, \frac{\sigma}{2}\right) $ and $p>1$. We consider the following special semi-linear Cauchy problem of wave and beam equations with structural damping and additional dispersion term
	\begin{equation}\label{2.1} 
	\begin{array}{lll}
	u_{tt}(t,x)+(-\Delta)^{\sigma}u(t,x)+(-\Delta)^{2\delta}u(t,x)+2(-\Delta)^{\delta}u_{t}(t,x)=|u(t,x)|^{p},
	\hfill&
	&\cr
	\\ \hspace{3cm}
	u(0,x)=0, \ \ u_{t}(0,x)=u_{1}(x). & 
	\end{array}  
	\end{equation}
	
	Let us assume suitable restrictions for the nonlinearity $p$ and the dimension $n$
	\begin{equation}\label{2.2} 
	\begin{array}{lll}
	2\leq p \leq \frac{n}{n-2\sigma} \hfill& {if } \ 2\sigma<n\leq 4\sigma,
	&\cr
	\\
	2\leq p < \infty & {if }\ 
	\sigma<n \leq 2\sigma.
	\end{array}  
	\end{equation}
	Moreover, we suppose
	\begin{equation}\label{2.3}  
	p>1+\frac{\sigma+2\delta}{n-\sigma}.
	\end{equation}\\
	Then, there exists a constant $\varepsilon_0>0$ such that for any small initial velocity 
	$ u_{1}\in L^{1}\cap L^{2} $\ with  $\|u_{1}\|_{L^{1}\cap L^{2}}<\varepsilon_0,$ 
	we have a uniquely determined globally (in time) solution to (\ref{2.1}) in the class
	$$u\in\mathcal{C}\left([0,\infty), H^{\sigma}\right)\cap \mathcal{C}^{1}\left([0,\infty),L^{2}\right).$$
  In addition, the unique solution and its energy satisfy the long time decay estimates as $t\to \infty$:
	\begin{align*}
	\|u(t,\cdot)\|_{L^{2}} &\lesssim (1+t)^{-\frac{n-2\sigma}{4\delta}}\|u_{1}\|_{L^{1}\cap L^{2}},
	\\ 
	\|(-\Delta)^{\sigma/2}u(t,\cdot)\|_{L^{2}} &\lesssim (1+t)^{-\frac{n}{4\delta}}\|u_{1}\|_{L^{1}\cap L^{2}},
	\\ 
	\|u_{t}(t,\cdot)\|_{L^{2}}&\lesssim(1+t)^{-\frac{n-(2\sigma-4\delta)}{4\delta}}\|u_{1}\|_{L^{1}\cap L^{2}}.
	\end{align*}
\end{theorem}
\begin{remark}
	The condition (\ref{2.2}) is technical due to the application of fractional Gagliardo-Nirenberg inequality from Lemma (\ref{FGN}), but it can be extended by relying on more general $L^{q}-L^{r}$ estimates with $1\leq q \leq r \leq \infty$. The lower bound (\ref{2.3}) is of important interest because it ensures that there is no loss of decay estimates for semi-linear problem with those for solutions to the corresponding linear Cauchy problem.
\end{remark} 

\begin{remark}
	The lower bound (\ref{2.3}) for global existence appears to combine the parameters of diffusion and oscillation observed in the Fourier multiplier
	$$m(t,\xi)=\frac{e^{-|\xi|^{2\delta}t} e^{\pm i|\xi|^{\sigma}t}}{2i|\xi|^{\sigma}}.$$
	But we are going to show that the upper bound for the blow up result do not depend at all on the parameter $\sigma$.
\end{remark}

Our second principal result concerning the blow up of weak solution is reads as follows.

\begin{theorem}
	Let $\sigma\geq 1$, $\delta \in (0, \frac{\sigma}{2})$ and $n> 2\delta$. We assume $u_{0}=0$ and $u_{1} \in L^{1}$ satisfies the following strict positivity:
	\begin{equation*}\label{..}
	\int_{\mathbb{R}^{n}} u_{1}(x)dx>0.
	\end{equation*}
	Additionally, we suppose the condition
	\begin{equation}\label{2.4}
	1<p<1+\frac{4\delta}{n-2\delta}.
	\end{equation}
	Then, there is no global (in time) weak solution to (\ref{2.1}).	
\end{theorem}
\begin{remark}
We avoid discussing the case $p=1+\frac{4\delta}{n-2\delta}$, because the main purpose of this research is just to understand the gap between the exponents of the existence and non-existence results.
\end{remark}
\section{Proof of Global Existence to (\ref{2.1})}\label{section 3}
In order to prove this result, let us collect some needed notations and inequalities. 

The notation $f\lesssim g$ is used to denotes that there exists a constant $c>0$ such that $f \leq cg$, while $f\approx g$ to denotes $g\lesssim f\lesssim g$, these constants have no significance in our analysis.

 The spaces $H^{s}(\mathbb{R}^{n})$ with $s>0$ denote Sobolev spaces of fractional order as defined below
$$H^{s}(\mathbb{R}^{n}):=\left\lbrace f\in S'(\mathbb{R}^{n}): \|f\|_{H^{s}(\mathbb{R}^{n})}=\|(1+|\cdot|^{2})^{\frac{s}{2}}\mathcal{F}(f)(\cdot)\|_{L^{2}(\mathbb{R}^{n})}<\infty\right\rbrace,$$
see \cite[p445]{EbertReissig} for more detail and reference therein.
\subsection{Useful Inequalities}
The fractional Gagliardo-Nirenberg inequality is expressed in the next lemma.
\begin{lemma}\label{FGN}
	Let $1<q<\infty$, $s>0$. Then, the following fractional Gagliardo-Nirenberg inequality holds for all $g\in H^{s}(\mathbb{R}^{n})$
	\[\|g\|_{L^{q}(\mathbb{R}^{n})}\lesssim \|(-\Delta)^{s/2}g\|_{L^{2}(\mathbb{R}^{n})}^{\theta_{q}}\,\|g\|_{L^{2}(\mathbb{R}^{n})}^{1-\theta_{q}},\]
	where \[\theta_{q}=\frac{n}{s}\left(\frac{1}{2}-\frac{1}{q}\right)\in\left[0,1\right].\]
\end{lemma}
\begin{proof}
	For the proof, see \cite{EbertReissig} and reference therein.
\end{proof}

The next integral inequality is used to deal with the Duhamel's integral. 
\begin{lemma}\label{Integral inequality}
	Let $a, b\in\mathbb{R}$ such that $\max\{a,b\}>1$. Then, it holds
	$$\int_{0}^{t}(1+t-\tau)^{-a}(1+\tau)^{-b}d\tau\lesssim (1+t)^{-\min\{a,b\}}.$$ 
\end{lemma}
\begin{proof}
	For the proof, see again \cite{EbertReissig} and reference therein.
\end{proof}

In the next subsection, we derive suitable energy estimates which are very important tool to demonstrate Theorem \ref{GlobalExistence1}.
\subsection{$L^{2}$-Energy Estimates}
	Let $t\in \mathbb{R}_{+}$, $x\in \mathbb{R}^{n}$, $\sigma\geq 1$ and $\delta\in \left(0, \frac{\sigma}{2}\right) $. We consider the corresponding linear equations to (\ref{2.1})
\begin{equation}\label{3.1} 
\begin{array}{lll}
u_{tt}(t,x)+2(-\Delta)^{\delta}u_{t}(t,x)+(-\Delta)^{2\delta}u(t,x)+(-\Delta)^{\sigma}u(t,x)=0, \ \ \ \hfill&
&\cr
\\ \hspace{3cm}
u(0,x)=0, \ \ u_{t}(0,x)=u_{1}(x). &.
\end{array}  
\end{equation}
The full symbol of the equation in (\ref{3.1}) is given after applying the space Fourier transform
$$\lambda^{2}+2|\xi|^{2\delta}\lambda+|\xi|^{4\delta}+|\xi|^{2\sigma}.$$

The characteristic roots are complex conjugate for all frequencies
$$\lambda_{1}(\xi)=-|\xi|^{2\delta}+i|\xi|^{\sigma}, \ \ \ \lambda_{2}(\xi)=-|\xi|^{2\delta}-i|\xi|^{\sigma},  \ \  \lambda_{1}(\xi)-\lambda_{2}(\xi)=2i|\xi|^{\sigma},$$
$$|\lambda_{1}(\xi)|^{2}=|\lambda_{2}(\xi)|^{2}=|\xi|^{4\delta}+|\xi|^{2\sigma}, \ \ |\lambda_{1}(\xi)-\lambda_{2}(\xi)|^{2}=4|\xi|^{2\sigma}.$$
We have the following $(L^{1}\cap L^{2})-L^{2}$ long time decay estimates.
\begin{proposition}\label{linear estimates}
	 The solution and its energy satisfy the following $(L^{1}\cap L^{2})-L^{2}$ estimates for all $n>2\sigma$
	 
	 	\begin{align*}
	 \|u(t,\cdot)\|_{L^{2}} &\lesssim (1+t)^{-\frac{n-2\sigma}{4\delta}}\|u_{1}\|_{L^{1}\cap L^{2}},
	 \\ 
	 \|(-\Delta)^{\sigma/2}u(t,\cdot)\|_{L^{2}} &\lesssim (1+t)^{-\frac{n}{4\delta}}\|u_{1}\|_{L^{1}\cap L^{2}},
	 \\ 
	 \|(-\Delta)^{\delta}u(t,\cdot)\|_{L^{2}}&\lesssim(1+t)^{-\frac{n-(2\sigma-4\delta)}{4\delta}}\|u_{1}\|_{L^{1}\cap L^{2}},
	 \\
	 \|u_{t}(t,\cdot)\|_{L^{2}}&\lesssim(1+t)^{-\frac{n-(2\sigma-4\delta)}{4\delta}}\|u_{1}\|_{L^{1}\cap L^{2}}.
	 \end{align*}
	In particular, if we remove the additional $L^{1}$ regularity, we have the $L^{2}-L^{2}$ estimates
	\begin{equation*}
	\|u(t,\cdot)\|_{L^{2}}\lesssim(1+t)\|u_{1}\|_{L^{2}}.
	\end{equation*}	
\end{proposition}
\begin{remark}
	The condition $n>2\sigma$ guarantees the decay property in the solution itself. We remark the equivalence between decay rates of the kinetic energy and elastic energy of order $\delta$. The kinetic energy satisfies worse decay rate with respect to elastic energy of order $\sigma/2$. If $\sigma<2\delta$, we have the same result as in \cite[Proposition 24]{pham-reissig-kainan}, but in this case, it is remarkable that the situation is opposite, that is, the kinetic energy satisfies the same decay rate with respect to elastic energy of order $\sigma/2$, and worse decay rate with respect to elastic energy of order $\delta$.
\end{remark}
\begin{proof}
	The proof of this proposition follows the same direct detailed computations found in \cite[Proposition 24]{pham-reissig-kainan} with small modifications due to the condition $2\delta<\sigma$.
	
	The explicit representation of solution to the linear equation is given by
	$$\hat{u}(t,\xi)=\frac{e^{t\lambda_{1}(\xi)}-e^{t\lambda_{2}(\xi)}}{\lambda_{1}(\xi)-\lambda_{2}(\xi)}\hat{u_{1}}(\xi)=\frac{e^{-|\xi|^{2\delta}t+i|\xi|^{\sigma}t}-e^{-|\xi|^{2\delta}t-i|\xi|^{\sigma}t}}{2i|\xi|^{\sigma}}\hat{u_{1}}(\xi), \ \ \ \xi\in \mathbb{R}^{n}, \ \ t>0.$$
	The condition $2\delta<\sigma$ shows that the scaling $\xi t^{1/2\delta} \to \tilde{\xi}$ is better than $\xi t^{1/\sigma} \to \tilde{\xi}$.
	
	 Using the following inequality
	$$\||\xi|^{a}e^{-t|\xi|^{b}}\|_{L^{2}(|\xi|<1)}\lesssim (1+t)^{-\frac{n+2a}{2b}}, \ \ b>0, \ n+2a>0,$$
	together with Persoval identity
	$$\|u(t,\cdot)\|_{L^{2}(\mathbb{R}^{n})}:=\|\hat{u}(t,\cdot)\|_{L^{2}(\mathbb{R}^{n})}$$
	give us the desired result
	$$\|\hat{u}(t,\cdot)\|_{L^{2}(|\xi|<1)}\lesssim \||\xi|^{-\sigma}e^{-t|\xi|^{2\delta}}\|_{L^{2}(|\xi|<1)}\|\hat{u_{1}}\|_{L^{\infty}}\lesssim(1+t)^{-\frac{n-2\sigma}{4\delta}}\|u_{1}\|_{L^{1}},$$
	$$\|\hat{u}(t,\cdot)\|_{L^{2}(|\xi|>1)}\lesssim \||\xi|^{-\sigma}e^{-t|\xi|^{2\delta}}\|_{L^{\infty}(|\xi|>1)}\|\hat{u_{1}}\|_{L^{2}}\lesssim e^{-ct}\|u_{1}\|_{L^{2}}.$$
\end{proof}

Let us now devote our attention to prove first principal result.
\begin{proof}	
	Since we are dealing with semi-linear Cauchy problems, we use the Banach's fixed point theorem inspired from the book \cite[Page 303]{EbertReissig}. Here, we need to define a family of evolution spaces $X(T)$ for any $T>0$ with suitable norm $\|\cdot\|_{X(T)}$, as well as an operator $$O :u\in X(T)\longmapsto Ou(t,x)=K(t,x)\ast u_{1}(x)+  \int_{0}^{t}K(t-\tau,x)\ast |u(\tau,x)|^{p}d\tau$$
	$$=\int_{\mathbb{R}^{n}}K(t,x-y)u_{1}(y)dy+\int_{0}^{t}\int_{\mathbb{R}^{n}}K(t-\tau,x-y)|u(\tau,y)|^{p}dyd\tau=u^{L}(t,x)+u^{N}(t,x),$$
	where 
	$$K(t,x)=\mathcal{F}^{-1}\left(\frac{e^{-|\xi|^{2\delta}t+i|\xi|^{\sigma}t}-e^{-|\xi|^{2\delta}t-i|\xi|^{\sigma}t}}{2i|\xi|^{\sigma}} \right)(t,x). $$
	If the operator $O$ satisfies the two inequalities: 
	\begin{align}
	\|Ou\|_{X(T)} &\lesssim \left\|u_{1}\right\|_{L^{1}(\mathbb{R}^{n})\cap L^{2}(\mathbb{R}^{n})}+\|u\|_{X(T)}^{p}, \ \  \forall u \in X(T), \label{3.2} \\ 
	\|Ou-Ov\|_{X(T)} &\lesssim \|u-v\|_{X(T)}\Big( \|u\|_{X(T)}^{p-1}+\|v\|_{X(T)}^{p-1}\Big), \ \  \forall u, v \in X(T), \label{3.3} 
	\end{align} 
	then, one can deduce the existence and uniqueness of a global (in time) solutions of (\ref{1.5}) for small norm of initial data. Here, the smallness of the initial data $\left\|u_{1}\right\|_{L^{2}(\mathbb{R}^{n})\cap L^{1}(\mathbb{R}^{n})}<\varepsilon_0$ imply that the operator $O$ maps balls of $X(T)$ into balls of $X(T)$.
	
	Now, we define the Banach space $X(T)$ for all $T>0$ as follows: 
	$$
	X(T):=\mathcal{C}\left([0,T], H^{\sigma}\right)\cap \mathcal{C}^{1}\left([0,T],L^{2}\right),
	$$
	we equip it with the norm 
	\begin{align}
	\|u\|_{X(T)}&=\sup_{0\leq t\leq T}\Big((1+t)^{\frac{n-2\sigma}{4\delta}}\|u(t,\cdot)\|_{L^{2}}+(1+t)^{ \frac{n}{4\delta}}\|(-\Delta)^{\sigma/2}u(t,\cdot)\|_{L^{2}} \nonumber \\
	&\hspace{3cm} 
	+(1+t)^{\frac{n-(2\sigma-4\delta)}{4\delta}}\|u_{t}(t,\cdot)\|_{L^{2}}\Big). \label{3.4}
	\end{align}
	The definition of this norm plays a crucial role in our analysis.
	
	\textit{Step 1:} Using energy estimates, it is clear that the function $$u^{L}(t,x)=K(t,x)\ast u_{1}(x)=\int_{\mathbb{R}^{n}}K(t,x-y)u_{1}(y)dy$$ belongs to $X(T)$ and we have
	\begin{align*}
	\|u^{L}\|_{X(T)}&=\sup_{0\leq t\leq T}\Big((1+t)^{\frac{n-2\sigma}{4\delta}}\|u^{L}(t,\cdot)\|_{L^{2}}+(1+t)^{ \frac{n}{4\delta}}\|(-\Delta)^{\sigma/2}u^{L}(t,\cdot)\|_{L^{2}} \nonumber \\
	&\hspace{3cm} 
	+(1+t)^{-\frac{n-(2\sigma-4\delta)}{4\delta}}\|u_{t}^{L}(t,\cdot)\|_{L^{2}}\Big)\lesssim \left\|u_{1}\right\|_{L^{1}(\mathbb{R}^{n})\cap L^{2}(\mathbb{R}^{n})}.	\end{align*} 
	\textit{Step 2:} To conclude inequality (\ref{3.2}), we must prove
	\begin{equation}\label{3.5}
	\|u^{N}\|_{X(T)}\lesssim \|u\|_{X(T)}^{p}.
	\end{equation}
	We divide the interval $[0,t]$ into two sub-intervals $[0,t/2]$ and $[t/2,t]$ where we use the $L^{1}-L^{2}$ linear estimates if $\tau\in[0,t/2]$ and $L^{2}-L^{2}$ estimates if $\tau\in[t/2,t]$. From Proposition \ref{linear estimates} we have
	\begin{align}
	\|u^{N}(t,\cdot)\|_{L^{2}}&\lesssim\int_{0}^{t/2}(1+t-\tau)^{-\frac{n-2\sigma}{4\delta}} \left\||u(\tau,\cdot)|^{p}\right\| _{L^{1}\cap L^{2}} d\tau \nonumber \\
	&\hspace{3cm} 
	+\int_{t/2}^{t}(1+t-\tau)\left\||u(\tau,\cdot)|^{p}\right\| _{L^{2}}d\tau.  \label{3.11}
	\end{align}
	By the fractional Gagliardo-Nirenberg inequality from Lemma \ref{FGN}, we can estimate these two norms 
	$$\|u(\tau,\cdot)\|^{p}_{L^{2p}},\ \ \|u(\tau,\cdot)\|^{p}_{L^{p}}.$$
	In fact, we know from (\ref{3.2}) that
	\begin{equation*}
	(1+\tau)^{\frac{n}{4\delta}}\|(-\Delta)^{\sigma/2}u(\tau,\cdot)\|_{L^{2}}\lesssim\|u\|_{X(T)},
	\end{equation*}
	\begin{equation*}
	(1+\tau)^{\frac{n-2\sigma}{4\delta}}\|u(\tau,\cdot)\|_{L^{2}}\lesssim\|u\|_{X(T)}.
	\end{equation*}
	Hence, we can estimate the above norms as follows
	\begin{equation*}
	\left\|u(\tau,\cdot)\right\| _{L^{sp}}^{p}\lesssim(1+\tau)^{-p(\frac{n-\sigma}{2\delta})+\frac{n}{2s\delta} }\|u\|_{X(T)}^{p}, \ s=1, 2,
	\end{equation*}
	provided that the conditions (\ref{2.2}) are satisfied for $p$ and $n$, we conclude 
	\begin{equation*}
	\left\|u(\tau,\cdot)\right\|^{p}_{L^{p}}+\left\|u(\tau,\cdot)\right\|^{p}_{L^{2p}}\lesssim (1+\tau)^{-p(\frac{n-\sigma}{2\delta})+\frac{n}{2\delta} }\|u\|_{X(T)}^{p}. 
	\end{equation*}
	Using the following equivalences
	$$(1+t-\tau)\approx (1+t) \ \text{if} \ \tau \in[0,t/2],\  \ (1+\tau)\approx (1+t) \ \text{if} \ \tau \in[t/2,t] $$
	and Lemma \ref{Integral inequality}, we estimates the first integral of $u^{N}$ over $[0,t/2]$ as follows
	$$
	\int_{0}^{t/2}(1+t-\tau)^{-\frac{n-2\sigma}{4\delta}}(1+\tau)^{-p(\frac{n-\sigma}{2\delta})+\frac{n}{2\delta} }\|u\|_{X(T)}^{p}d\tau \lesssim(1+t)^{-\frac{n-2\sigma}{4\delta}}\|u\|_{X(T)}^{p}\int_{0}^{t/2}(1+\tau)^{-p(\frac{n-\sigma}{2\delta})+\frac{n}{2\delta} }d\tau$$
	$$
	\lesssim(1+t)^{-\frac{n-2\sigma}{4\delta}}\|u\|_{X(T)}^{p}
	$$
	due to $p>1+\frac{\sigma+2\delta}{n-\sigma}$. For the second integral over $[t/2,t]$ we have
	
	$$\int_{t/2}^{t}(1+t-\tau)(1+\tau)^{-p(\frac{n-\sigma}{2\delta})+\frac{n}{4\delta}}\|u\|_{X(T)}^{p}d\tau\lesssim(1+t)^{2-p(\frac{n-\sigma}{2\delta})+\frac{n}{4\delta} }\|u\|_{X(T)}^{p}.$$
	Thanks to $2\delta<\sigma$ and $p>1+\frac{\sigma+2\delta}{n-\sigma}$, we arrive to the desired estimate for $u^{N}$ 
	$$
	(1+t)^{\frac{n-2\sigma}{4\delta}}\|u^{N}(t,\cdot)\|_{L^{2}}\lesssim\|u\|_{X(T)}^{p}.
	$$
	The same conditions allow us to obtain
	$$ 
	(1+t)^{\frac{n-(2\sigma-4\delta)}{4\delta}}\|u_{t}^{N}(t,\cdot)\|_{L^{2}}\lesssim\|u\|_{X(T)}^{p},
	$$
	$$
	(1+t)^{\frac{n}{4\delta}}\|(-\Delta)^{\sigma/2}u^{N}(t,\cdot)\|_{L^{2}}\lesssim\|u\|_{X(T)}^{p}.
	$$ 
	Inequality (\ref{3.5}) is now proved, that is (\ref{3.2}).
	
	\textit{Step 3:} To prove (\ref{3.3}) we choose two elements $u$, $v$ belong to $X(T)$, and we write
	$$Ou(t,x)-Ov(t,x)=\int_{0}^{t}K(t-\tau,x)\ast (|u(\tau,x)|^{p}-|v(\tau,x)|^{p})d\tau.$$
	We divide the integral as above and by the H\"{o}lder's inequality, we derive for $s=1,2$ the following
	\begin{equation*}
	\left\||u(\tau,\cdot)|^{p}-|v(\tau,\cdot)|^{p}\right\|_{L^{s}}\leq \|u(\tau,\cdot)-v(\tau,\cdot)\|_{L^{sp}}\left(\|u(\tau,\cdot)\|_{L^{s p}}^{p-1}+\|v(\tau,\cdot)\|_{L^{sp}}^{p-1} \right).
	\end{equation*}
	Using the definition of the norm $\|u-v\|_{X(T)}$ and fractional Gagliardo-Nirenberg inequality we can prove (\ref{3.3}) without difficulty. Hence, Theorem \ref{GlobalExistence1} is proved.	
\end{proof}
\section{Proof of Blow up Result to (\ref{2.1})}\label{section 4}
The proof will follows the same detailed computations found in \cite[Theorem 2]{dao-reissig-blow up} with small modifications due to the additional dispersion term. Therefor, we will only present the steps in which we have a difference. We collect the following auxiliary lemmas whose proofs can be found in \cite{dao-reissig-blow up}. 

\begin{lemma}\label{lemma 4.1}
	Let $d \in \mathbb{N}$ and $s \in [0,1)$. Let $\langle x\rangle=(1+|x|^{2})^{1/2}$. Then, the following estimates hold for any $\nu>n$ and for all $x \in \mathbb{R}^{n}$:
	\begin{align*}
	\left|(-\Delta)^{d+s} \langle x\rangle^{-\nu}\right| \lesssim
	\begin{cases}
	\langle x\rangle^{-n-2d} & \text{if}\ \  s=0, \\
	\langle x\rangle ^{-n-2s} & \text{if}\ \  s \in (0,1).
	\end{cases}
	\end{align*}
\end{lemma}
\begin{lemma}\label{lemma 4.2}
	Let $\phi:= \phi(x)= \langle x\rangle^{-\nu}$ for some $\nu>0$. For any $R>0$ and for some constant $\theta>0$, let $\phi_R$ be a function defined by
	$$ \phi_R(x)= \phi(x/R^\theta) \quad \text{ for all }\ \ x \in \mathbb{R}^{n}. $$
	Then, $(-\Delta)^\rho (\phi_R)$ with $\rho>0$ satisfies the following scaling properties for all $x \in \mathbb{R}^{n}$:
	\begin{equation*}
	(-\Delta)^\rho (\phi_R)(x)= R^{-2\rho\theta} \big((-\Delta)^\rho \phi \big)(x/R^\theta).
	\end{equation*}
\end{lemma}
\begin{lemma}\label{lemma 4.3}
	Let $s \in \mathbb{R}$. Let $f=f(x) \in H^s(\mathbb{R}^{n})$ and $g=g(x) \in H^{-s}(\mathbb{R}^{n})$. Then, the following relation holds:
	$$ \int_{\mathbb{R}^{n}}f(x)\,g(x)dx= \int_{\mathbb{R}^{n}}\ \hat{f}(\xi)\,\hat{g}(\xi)d\xi. $$
\end{lemma}
Using Lemma \ref{lemma 4.3}, we may obtain the following exchange property
\begin{align*}
\int_{\mathbb{R}^{n}}(-\Delta)^{\sigma}u(t,x)\phi(x) \mathrm{d}x &=\int_{\mathbb{R}^{n}}u(t,x) (-\Delta)^{\sigma}\phi_R(x) \mathrm{d}x, \\ 
\int_{\mathbb{R}^{n}}(-\Delta)^{2\delta}u(t,x) \phi(x) \mathrm{d}x &=\int_{\mathbb{R}^{n}}u(t,x) (-\Delta)^{2\delta}\phi(x) \mathrm{d}x.
\\ 
\int_{\mathbb{R}^{n}}(-\Delta)^{\delta}u_{t}(t,x)\phi(x) \mathrm{d}x &=\int_{\mathbb{R}^{n}}u_{t}(t,x) (-\Delta)^{\delta}\phi(x)\mathrm{d}x.
\end{align*}

 Let us first define the radial space-dependent test function
$$\phi(x)=(1+|x|^{2})^{-n-2\delta}, \ \ x\in \mathbb{R}^{n},$$ and we choose the time-dependent test function $\eta=\eta(t) \in \mathcal{C}_0^{\infty}\big([0,\infty)\big)$ fulfilling
\begin{align*}
\eta(t):=\begin{cases}
1&\mbox{if} \ \ 0\le t\le 1/2,\\
\mbox{decreasing}&\mbox{if}\ \  1/2\le t\le 1,\\
0&\mbox{if} \ \ 1\leq t\leq \infty,
\end{cases}
\end{align*}
and
\begin{align}\label{4.1}
\big(\eta(t)\big)^{-\frac{p'}{p}}\left(|\eta'(t)|^{p'}+|\eta''(t)|^{p'}\right)\le C\quad \mbox{for any}\ \ t\in [1/2,1],
\end{align}
where $p'$ is the conjugate of power nonlinearity $p>1$.  We first give a definition of weak solution to (\ref{2.1}).
\begin{definition}
Let $p>1$. We say that $u\in L^{p}_{loc}([0,\infty)\times \mathbb{R}^{n})$ is a global weak solution to (\ref{2.1}), if, for any test function $ \psi \in \mathcal{C}^{\infty}_{c}([0, \infty) \times \mathbb{R}^{n})$, it holds:
\begin{align*}
\int_0^{\infty}\int_{\mathbb{R}^{n}}|u(t,x)|^p\psi(t,x)\mathrm{d}x\mathrm{d}t &=\int_{0}^{\infty}\int_{\mathbb{R}^{n}}u(t,x)\psi_{tt}(t,x)\,\mathrm{d}x\mathrm{d}t  + \int_0^{\infty}\int_{\mathbb{R}^{n}}(-\Delta)^{\sigma}\psi(t,x)\,u(t,x)\,\mathrm{d}x\mathrm{d}t
 \nonumber \\
&\hspace{-3cm}+ \int_0^{\infty}\int_{\mathbb{R}^{n}} (-\Delta)^{2\delta}\psi(t,x)\,u(t,x)\,\mathrm{d}x\mathrm{d}t-2 \int_{0}^{\infty}\int_{\mathbb{R}^{n}} (-\Delta)^{\delta}\psi_{t}(t,x)\,u(t,x)\,\mathrm{d}x\mathrm{d}t -\int_{\mathbb{R}^{n}} u_1(x)\psi(0,x)\,\mathrm{d}x.
 \nonumber
\end{align*}
If $supp\,\psi \subset [0, T]\times \mathbb{R}^{n}$ for some finite $0<T< \infty$, then $u$ is said to be a local weak solution to (\ref{2.1}).
\end{definition}

  Let $R$ be a large parameter in $[0,\infty)$, we fix the following R-dependent test function:
$$ \psi_R(t,x):= \eta_R(t)\phi_R(x)= \eta(t/R^{\alpha})\phi(x/R), $$
where $\alpha$ will be determined later to catch our result. We also define the functionals
\begin{align*}
I_R:=\int_0^{\infty}\int_{\mathbb{R}^{n}}|u(t,x)|^p\psi_R(t,x)\mathrm{d}x\mathrm{d}t=\int_0^{R^{\alpha}}\int_{\mathbb{R}^{n}}|u(t,x)|^p\psi_R(t,x)\mathrm{d}x\mathrm{d}t.
\end{align*}
and 
\begin{align*}
I_{R,t}:=\int_{R^{\alpha}/2}^{R^{\alpha}}\int_{\mathbb{R}^{n}}|u(t,x)|^p\psi_R(t,x)\mathrm{d}x\mathrm{d}t.
\end{align*}
The idea of this approach is based on contradiction argument. In fact, let us assume that we have a global (in time) Sobolev solution from $\mathcal{C}([0,\infty), L^{2}) $ to (\ref{2.1}). Thanks to the support condition for $\eta_R(t)$ and the assumption $u(0,x)=0$, we carry out integration by parts we arrive at the following equation( local weak solution):
\begin{align}
0\le I_R &= -\int_{\mathbb{R}^{n}} u_1(x)\phi_R(x)\,\mathrm{d}x + \int_{R^{\alpha}/2}^{R^{\alpha}}\int_{\mathbb{R}^{n}}u(t,x) \eta''_R(t) \phi_R(x)\,\mathrm{d}x\mathrm{d}t \nonumber \\
&\quad + \int_0^{\infty}\int_{\mathbb{R}^{n}} \eta_R(t) \phi_R(x)\, (-\Delta)^{\sigma} u(t,x)\,\mathrm{d}x\mathrm{d}t \nonumber \\
&\quad + \int_0^{\infty}\int_{\mathbb{R}^{n}} \eta_R(t) \phi_R(x)\, (-\Delta)^{2\delta} u(t,x)\,\mathrm{d}x\mathrm{d}t \nonumber \\
&\quad -2 \int_{R^{\alpha}/2}^{R^{\alpha}}\int_{\mathbb{R}^{n}} \eta'_R(t) \phi_R(x)\, (-\Delta)^{\delta} u(t,x)\,\mathrm{d}x\mathrm{d}t \nonumber \\
&:= -\int_{\mathbb{R}^{n}} u_1(x)\varphi_R(x)\,dx+ J_{1,R}+ J_{2,R}+J_{3,R}-J_{4,R}. \label{4.2}
\end{align}

Repeating the same computations as in the cited paper \cite{dao-reissig-blow up}, we may now arrive at the following estimates
$$|J_{1,R}|\lesssim I_{R,t}^{1/p}R^{-2\alpha+\frac{n+\alpha}{p'}}, \ \ |J_{2,R}|\lesssim I_{R}^{1/p}R^{-2\sigma+\frac{n+\alpha}{p'}},$$

$$|J_{3,R}|\lesssim I_{R}^{1/p}R^{-4\delta+\frac{n+\alpha}{p'}}, \ \  |J_{4,R}|\lesssim I_{R,t}^{1/p}R^{-\alpha-2\delta+\frac{n+\alpha}{p'}}.$$
Combining these estimates we get
$$0<\int_{\mathbb{R}^{n}} u_1(x)\phi_R(x)\,\mathrm{d}x \lesssim I_{R,t}^{1/p}\left(R^{-2\alpha+\frac{n+\alpha}{p'}}+ R^{-\alpha-2\delta+\frac{n+\alpha}{p'}}\right)+I_{R}^{1/p}\left(R^{-2\sigma+\frac{n+\alpha}{p'}}+ R^{-4\delta+\frac{n+\alpha}{p'}}\right)-I_{R}.$$
The best possible choice of parameter $\alpha$ is $\alpha=2\delta$, in this case, the above inequality becomes 
$$\lesssim I_{R}^{1/p} R^{-4\delta+\frac{n+2\delta}{p'}}-I_{R}.$$
Applying the crucial inequality for all $p>1$
$$ A\,y^{1/p}- y \le A^{\frac{p}{p-1}} \quad \text{ for any } A>0,\, y \ge 0 \text{ and } 0< 1/p< 1, \ \ p/(p-1)=p'$$
which leads to
$$0<\int_{\mathbb{R}^{n}} u_1(x)\phi_R(x)\,\mathrm{d}x \lesssim R^{-4\delta p'+n+2\delta}.$$
Now, let us suppose 
$$-4\delta p'+n+2\delta<0.$$ 
If we let $R\to \infty$ we get 
$$\int_{\mathbb{R}^{n}} u_1(x)\mathrm{d}x=0,$$
and this is contradict to our assumption. Hence, every Sobolev solution blows up if $p<1+\frac{4\delta}{n-2\delta}$. The proof is now completed.
\section{Conclusion}
Let us go back and consider the following Fourier multiplier 
$$m(t,\xi):=\frac{e^{-|\xi|^{2\delta}t} e^{\pm i|\xi|^{\sigma}t}}{2i|\xi|^{\sigma}}, \ \ \xi \in \mathbb{R}^{n}, \ \ t>0.$$ 
It is reasonable to distinguish between the two cases $2\delta<\sigma$ and $\sigma<2\delta$ in model (\ref{1.8}). One can see that the model (\ref{1.9}) never satisfies the first case. 
\begin{itemize}
	\item  if $\sigma<2\delta$, the oscillatory part satisfies better scaling than the diffusive part as in model (\ref{1.9}). Unfortunately, this benefit disappears when deriving $L^{2}$-energy estimates, and this later estimate is not sufficient to get sharp results. For this reason, the authors in \cite{dabbicco-ebert-lp-lq,dabbicco-ebert-critical exponent} employed another strategy based on $L^{q}$-energy estimates in time-dependent frequencies to deal with better scaling coming from oscillatory part $|\xi|^{-\sigma}e^{\pm i|\xi|^{\sigma}t}$, and they succeed to catch partial sharp result.
	\item if $2\delta<\sigma$, we have no benefit from the oscillatory part in deriving the $L^{2}$-energy estimates. Unfortunately, the benefit from diffusive part is not sufficient to get sharp results. In this case, we also remark that we have no explicit Fourier multiplier that leads to the sub-critical condition (\ref{2.4}). 
\end{itemize}
All in all, in our forthcoming paper, we will discuss more general $L^q$ $(1\leq q\leq \infty)$-energy estimates in which oscillations play a role and show whether or not the above gap may be improved, as done in the studies \cite{dabbicco-ebert-critical exponent,dabbicco-ebert-lp-lq}. Question 1.2 is now unanswered.

\subsubsection*{Acknowledgments}
The authors would like to thank the referees for their constructive comments and
suggestions that helped to improve the original manuscript. We thank Directorate-General for Scientific Research and Technological Development in Algeria (DGRSDT) for providing research conditions.

\vskip+0.5cm

\noindent \textbf{Authors address:}
\vskip+0.3cm

\noindent \textbf{Khaldi Said and Hakem Ali}

Laboratory of Analysis and Control of PDEs, Djillali Liabes University,  P.O. Box 89, Sidi-Bel-Abb\`{e}s 22000, Algeria

{\itshape E-mails:} \texttt{saidookhaldi@gmail.com, said.khaldi@univ-sba.dz, hakemali@yahoo.com}
\vskip+0.3cm

\noindent \textbf{Arioui Fatima Zahra}

Laboratory of Statistics and Stochastic Processes, Djillali Liabes University,  P.O. Box 89, Sidi-Bel-Abb\`{e}s 22000, Algeria

{\itshape E-mail:} \texttt{ariouifatimazahra@gmail.com}

\end{document}